\documentclass{amsart}
\usepackage{geometry}
\usepackage{amssymb, enumerate, xspace, graphicx}
\usepackage{xcolor}
\usepackage[all]{xy}
\usepackage{pgf}
\usepackage{scrextend}
\SelectTips{cm}{}
\usepackage{tikz}
\usepackage[T1]{fontenc}
\usepackage{lmodern}
\usepackage[final]{microtype}
\usetikzlibrary{arrows,matrix,positioning}
\usepackage{tikz-cd}
\usepackage{comment}
\usepackage{hyperref}
\excludecomment{toexclude} 
\usepackage{mathtools}

\numberwithin{equation}{section}

\numberwithin{subsection}{section}

\allowdisplaybreaks[1]

\newenvironment{enumeratei}
{\begin{enumerate}[\upshape (i)]}
{\end{enumerate}}

\newenvironment{enumerate1}
{\begin{enumerate}[\upshape (1)]}
{\end{enumerate}}

\newtheorem*{namedtheorem}{\theoremname}
\newcommand{\theoremname}{testing}

\newtheorem{theo}{Theorem}[section]
\newtheorem{prop}[theo]{Proposition}

\newtheorem{prop-def}[theo]
{Proposition-Definition}
\newtheorem{coro}[theo]{Corollary}
\newtheorem{lemm}[theo]{Lemma}
\newcommand \fr{\operatorname{F}}
\theoremstyle{definition}
\newtheorem{defi}[theo]{Definition}

\newtheorem{rema}[theo]{Remark}

\theoremstyle{remark}

\DeclareMathOperator{\Lie}{Lie}
\DeclareMathOperator{\IM}{Im}
\DeclareMathOperator{\Id}{Id}

\newcommand\cA{\mathcal{A}} 
 \newcommand\cD{\mathcal{D}}
\newcommand\cE{\mathcal{E}} 
 \newcommand\cH{\mathcal{H}}
\newcommand\cI{\mathcal{I}}

\renewcommand\AA{\mathbb{A}}

\newcommand\GG{\mathbb{G}}

 \newcommand\PP{\mathbb{P}}

\newcommand\mto{\ifinner\mapsto\else\longmapsto\fi}
\newcommand\too{\longrightarrow}

\newcommand \In{\subseteq}

\def\displaytimes_#1{\mathrel{\mathop{\times}\limits_{#1}}}

\def\displayotimes_#1{\mathrel{\mathop{\bigotimes}\limits_{#1}}}

\newcommand\GL{\operatorname{GL}}
\newcommand\PGL{\operatorname{PGL}}

\newcommand\spec{\operatorname{Spec}}

\newcommand\SL{\operatorname{SL}}

\newcommand\id{\mathrm{id}}

\newcommand{\benu}{\begin{enumerate}}
\newcommand{\eenu}{\end{enumerate}}

\DeclareMathOperator{\Hom}{Hom}

\newcommand{\on}{\stackrel}

\title{Unexpected subgroup schemes of $\PGL_{2,k}$ in characteristic $2$}

\author{Bianca Gouthier}
\author{Dajano Tossici}
\begin{document}
\maketitle
\begin{abstract}
    If the characteristic of a field $k$ is odd any infinitesimal group scheme of $\PGL_{2,k}$ lifts to $\SL_{2,k}$. In this paper, we prove that this is not true in characteristic $2$ and we give a complete description, up to isomorphism, of infinitesimal unipotent subgroup schemes of $\PGL_{2,k}$. Also, the infinitesimal trigonalizable case is considered. 
\end{abstract}

\section{Introduction}

In this paper, we are interested in finite subgroup schemes of $\PGL_{2,k}$.
In his paper \cite{Beauville}, Beauville classified, up to conjugacy, all finite subgroups of $\PGL_2(k)$ of order coprime with the characteristic. Here we are interested in the opposite case, infinitesimal subgroup schemes.
It seems to the authors that it is quite an accepted fact that any infinitesimal subgroup scheme of $\PGL_{2,k}$ lifts to $\GL_{2,k}$. In particular any unipotent infinitesimal sugbroup scheme of  $\PGL_{2,k}$ would be a subgroup scheme of $\GG_{a,k}$, and so it would be isomorphic to $\alpha_{p^n,k}$ for some $n\ge 0$. In this paper, we prove that this is not true if the characteristic of the field is $2$. The result is instead true if the characteristic is odd and 
we give a proof of it in the section \S \ref{Proof p>2}.

We recall that, for any field $k$, $\PGL_{2,k}$ represents the automorphism group functor of $\PP^1_k$. So the study of subgroup schemes corresponds to faithful actions on $\PP^1_k$.
Moreover $\PGL_{2,k}(k)$  coincides with the Cremona group in dimension one, i.e. birational morphisms of $\PP^1_k$, since any rational morphism from a projective nonsingular curve extends to the whole curve. In positive characteristic, the situation is completely different if we consider rational actions of infinitesimal group schemes. Most of the faithful infinitesimal actions of the affine line do not extend to $\PP^1_k$. For instance, all the actions of $\alpha_p^n$, with $n\ge 4$, over $\AA^1_k$ do not extend to $\PP^1_k$, since $\PGL_{2,k}$ has dimension $3$ and the Lie algebra of $\alpha_p^n$ has dimension $n$. See, for instance, \cite[Lemma 3.6]{brion2022actions} and \cite[Corollary 6.9]{gouthier2023infinitesimal}.

The main result of the paper is the following.



\begin{theo}
Let $k$ be a field of characteristic $2$.
\begin{enumerate1}
\item 
The infinitesimal unipotent subgroup schemes of $\PGL_{2,k}$ are, up to isomorphism, all and only the subgroup schemes of the semi-direct product $\alpha_{2^n,k}\rtimes \alpha_{2,k}$, with $n\ge 1$, where the action of $\alpha_{2,k}$ on $\alpha_{2^n,k}$ is given by $a\cdot b= b+ab^2$. 

\item If $k$ is perfect,
any infinitesimal trigonalizable, not unipotent, subgroup scheme of $\PGL_{2,k}$ is isomorphic to $\mu_{2^l,k}$ or to the semi-direct product of $\mu_{2^l,k}$, for some $l\ge 1$, by one of the two unipotent group schemes
\begin{enumeratei}
\item the semi-direct product  $\alpha_{2^n,k}\rtimes \alpha_2$, with $n\ge 1$, where the action of $\alpha_{2,k}$ on $\alpha_{2^n,k}$ is given by $a\cdot b= b+ab^2$
\item $\alpha_{2^n,k}$
\end{enumeratei}

for some nontrivial action of $\mu_{2^l,k}$.
\end{enumerate1}
\end{theo}

An explicit description of all these group schemes will be given in the section \S\ref{sec: noncommutative group schemes}. 
While the above Theorem gives a complete classification of infinitesimal unipotent subgroup schemes of $\PGL_{2,k}$, for trigonalizable group schemes 
we do not know if, for any nontrivial action of $\mu_{2^l,k}$ over the unipotent group schemes in $(i)$, the associated semi-direct product acts faithfully on $\PP^1_k$. In the section \S \ref{sec: trigonalizable}, we prove that there exists at least one action of $\mu_{2^l,k}$ over any unipotent group scheme which appears in $(i)$ such that the associated semi-direct product acts faithfully on $\PP^1_k$.
In the commutative case,  we get a complete classification over an algebraically closed field.


\begin{coro}
Let $k$ be an algebraically closed field of characteristic $2$.
The list of infinitesimal commutative subgroup schemes of $\PGL_{2,k}$, up to isomorphism, is the following:
\begin{enumerate1}
\item $\alpha_{2^n,k}$, for some $n\ge 0$,
\item $\alpha_{2,k}\times \alpha_{2,k}$,
\item the $2$-torsion of a supersingular elliptic curve,
\item $\mu_{2^n}$, for some $n>0$.
\end{enumerate1}
\end{coro}

The corollary follows from the Theorem using the Lemma \ref{lem:commutative subgroups}.

\subsection*{Acknowledgements:} The authors are deeply grateful to Fabio Bernasconi, Alice Bouillet, Andrea Fanelli, Pascal Fong and Matthieu Romagny for very useful conversations.

\section{Infinitesimal subgroup schemes of $\PGL_{2,k}$ in characteristic $p>2$} \label{Proof p>2} Let $k$ be a field of characteristic $p$.
Let $G$ be an infinitesimal subgroup scheme of $\PGL_{2,k}$. Since $\SL_{2,k}\to \PGL_{2,k}$ is an étale covering then $G$ lifts to $\SL_{2,k}$. This result is known and, for instance, it is mentioned in Fakhruddin \cite{Fakhruddin}.
We however report here the details of the proof, which are not present in the aforementioned paper. We have the following lemma.

\begin{lemm}
Any extension of group schemes
$$
1\to H\to G\on{\pi}{\to} Q\to 1,
$$
such that $H$ is étale and commutative and $Q$ is infinitesimal,
 is trivial and $G$ is isomorphic to the direct product $H\times Q$.
\end{lemm}

\begin{proof}
First of all, we observe that any action of $Q$ on $H$ is trivial.
Then the result follows from \cite[III \S 6, n.7]{DG}. 
We give here the proof in this simpler case.
Since $H$ is étale and $Q$ is infinitesimal then there is a schematic section of $\pi$. So all such extensions are classified by the Hochschild cohomology group $H_0^2(Q,H)$
(see \cite[III \S 6, n.7]{DG}).  Since $Q$ is infinitesimal and $H$ is étale any scheme morphism from $Q^2\to H$ is constant, so $H_0^2(Q,H)$ is trivial.
\end{proof}

\begin{prop}
If $k$ is a field of characteristic $p$ and $p$ does not divide $n$ then any infinitesimal subgroup scheme of $\PGL_{n,k}$ lifts to $\SL_{n,k}$.
\end{prop}
\begin{proof}
We consider the exact sequence
$$
1\to \mu_{n,k} \to \SL_{n,k} \on{\pi}{\to} \PGL_{n,k}\to 1.
$$
Since $p$ does not divide $n$ then $\mu_{n,k}$ is an étale group scheme. Let $G$ be an infinitesimal subgroup scheme of $\PGL_{n,k}$. Then $\pi^{-1}G$ is an extension of $G$ by $\mu_{n,k}$. By the previous Lemma the extension is trivial, so $G$ lifts to $\SL_{n,k}$.
\end{proof}
In particular the above Proposition applies when $n=2$ and $p>2$. 
\section{Infinitesimal unipotent subgroups schemes of $\GL_{2,k}$}
In this subsection, we give an explicit description of all infinitesimal unipotent subgroup schemes of $\GL_{2,k}$, where $k$ is a field of positive characteristic $p$.  The following result will be used in the proof of the Theorem. 

\begin{prop}\label{prop: subgroups SL2}
Any infinitesimal unipotent subgroup scheme of $\GL_{2,k}$ is one of the following subgroup schemes of $\SL_{2,k}$
$$
H_{s_1,s_2,n}=\{\left( \begin{smallmatrix}  x_{11} &x_{12} \\ x_{21} & x_{22}  \end{smallmatrix} \right)\in \ker \fr^n_{\SL_{2,k}}| s_i (x_{ii}-1)+s_j(x_{ij})=0, x_{22}=2-x_{11} \text{ for } (i,j)=(1,2),(2,1) \},
$$
for some $[s_1:s_2]\in \PP^1(k)$ and $n\ge 1$.
\end{prop}

\begin{proof}
Any infinitesimal unipotent subgroup scheme of $\GL_{2,k}$ is isomorphic to $\alpha_{p^n,k}$, for some $n$, since, up to conjugation, it is contained in the subgroup of upper triangular unipotent matrices, which is isomorphic to $\GG_{a,k}$.
Moreover any unipotent subgroup scheme $H$ of $\GL_{2,k}$  is contained in $\SL_{2,k}$ since, over a field, any map from a unipotent group scheme to a diagonalizable group scheme is trivial, so the restriction of the determinant to $H$ is trivial. We now observe that
$$
\Hom_{k-gr}(\alpha_{p^n,k}, \SL_{2,k})\In \SL_{2,k}(k[T]/(T^{p^n})) 
$$
and it consists of matrices 
$A(\overline{T})$ such that $A(\overline{S+T})=A(\overline{S})A(\overline{T}) \in \SL_{2,k}(k[S,T]/(S^{p^n},T^{p^n})$.
Since $A(0)=\Id$ we have that 
$$
A(\overline{T})=\sum_{i=0}^{p^n-1} A_i \overline{T}^i
$$
with $A_i\in M_2(k)$ for any $0\le i\le p^{n}-1$ and $A_0=\Id$.
Now
$$
A(\overline{S+T})=\sum_{i=0}^{p^n-1}\sum_{j=0}^i\binom{i}{j} A_i\overline{S}^j\overline{T}^{i-j} 
$$
and
$$
A(\overline{S})A(\overline{T})=\sum_{i,j=0}^{p^n-1}A_iA_{j}\overline{S}^i\overline{T}^{j}.
$$
Therefore for any $0\le i,j <p^n$ we have that
$$
\binom{i+j}{j} A_{i+j}=A_iA_{j}
$$
where we set $A_k=0$ if $k\ge p^n$.
Then 
$$
A_{i}^p=\binom{2i}{i}\cdots  \binom{pi}{i} A_{pi}
= \frac{(pi)!}{(i!)^p}A_{pi}$$ for any $0\le i\le p^n-1$. 
Now
$$
v_p\Big(\frac{(pi)!}{{(i!)}^p}\Big)=i-(p-1)v_p(i!)>0,
$$
so
$$
A_i^p=0.
$$
Moreover $A_{i}$ commutes with $A_{j}$ for any $0\le i,j \le p^n-1$. 
It is known that if two nilpotent matrices of rank $2$ commute then one is a multiple of the other.
Therefore there exists a nilpotent matrix $B\in M_2(k)$ and $f(\overline{T})\in k[T]/(T^{p^n})$ such that
$A(\overline{T})=\Id+ f(\overline{T}) B$. Now it is easy to verify that $f(\overline{T})$ is additive. 
Then the matrix $A(\overline{T})$ belongs to $H_{s_1,s_2,n}(k[T]/(T^{p^n}))$, where $(s_1,s_2)\in \ker B\setminus \{(0,0)\}$. 
So any infinitesimal unipotent subgroup scheme of $\GL_{2,k}$ is contained in some $H_{s_1,s_2,n}$ for some $[s_1:s_2]\in \PP^1(k)$.
On the other hand, for any  $[s_1:s_2]\in \PP^1(k)$ the matrix
$$
A(\overline{T})=\Id+\overline{T}\left( \begin{smallmatrix} s_1s_2& -s_1^2\\  s_2^2& -s_1s_2  \end{smallmatrix}\right )
$$
gives an isomorphism between $\alpha_{p^n,k}$ and $H_{s_1,s_2,n}$.

\end{proof}

\section{Some noncommutative unipotent infinitesimal group schemes.}  \label{sec: noncommutative group schemes}
In this section we explicitly describe the group schemes which appear in the Theorem.
\begin{defi}
Let $k$ be a field of characteristic $p$ and consider the action of $\alpha_{p,k}$, as an automorphism of groups, on $\GG_{a,k}$ given by $a\cdot b= b+ab^p$.
We define the associated semi-direct product $\cE=\GG_{a,k}\rtimes \alpha_{p,k}$.
\begin{enumeratei}
\item For any $n\ge 0$ we define the subgroup scheme $\cD_n$ of $\cE$ induced by the  closed immersion $\alpha_{p^n,k}\to \GG_{a,k}$. It is isomorphic to the induced semi-direct product $\alpha_{p^n}\rtimes \alpha_{p,k}$.
\item For any $a\in k$ and $ n\ge 1$, we define   $\mathcal{H}_{a,n}$ as the kernel of the morphism
$$
(-a\fr^{n-1}, i):\cD_{n}\to \GG_{a,k}
$$
where $i$ is the inclusion $i:\alpha_{p,k}\to \GG_{a,k}$.

\end{enumeratei}
\end{defi}

Explicitly we have that $\cD_n$ is isomorphic to $\spec(k[S,T]/(S^p,T^{p^n})$
where  the comultiplication is given by 
 $$
\overline{S}\mapsto \overline{S}\otimes 1+1\otimes \overline{S}
$$ 
and
$$
\overline{T}\mapsto  \overline{T}\otimes 1+1\otimes \overline{T} +\overline{S} \otimes \overline{T}^p.
$$
We observe that, for $p=2$, this is the group that appears in the statement of the Theorem. Indeed the previous group scheme is equally isomorphic to 
$\spec k[S,T]/(S^{p},T^{p^{n}})$, where the comultiplication is given by
$$
\overline{S}\mapsto \overline{S}\otimes 1+1\otimes \overline{S}
$$ 
and
$$
\overline{T}\mapsto  \overline{T}\otimes 1+1\otimes \overline{T} +\overline{T}^p \otimes \overline{S}.
$$
 The two group schemes are isomorphic via $
\overline{S}\mapsto \overline{S}, \overline{T}\mapsto  \overline{T}+\overline{S}\overline{T}^p. $ We will need this second presentation in the proof of the Theorem.

Moreover $\cH_{a,n}$ is isomorphic to $\spec k[T]/(T^{p^{n}})$ where the comultiplication is given by
$$
\overline{T}\mapsto  \overline{T}\otimes 1+1\otimes \overline{T} +a\overline{T}^{p^{n-1}} \otimes \overline{T}^{p}.
$$
We collect some easy results that will be freely used in the rest of the paper.
\begin{lemm} \label{lem:commutative subgroups} Let $k$ be a field of characteristic $p$.
    \begin{enumerate1}
\item        
 $\cD_0$ is isomorphic to $\alpha_{p,k}$ and $\cD_1$ is isomorphic to $\alpha_{p,k}\times\alpha_{p,k}$. 

\item For any $n\ge 2$, $\cH_{0,1}$ is the center of $\cD_n$ and $\cD_{n}/\cH_{0,1}$ is isomorphic to $\alpha_{p^{n-1},k}\times \alpha_{p,k}$.

\item $\cH_{a,n}$ is commutative if and only if $n\le 2$ or $a=0$.
\item $\cH_{0,n}$ is isomorphic to $\alpha_{p^n,k}$ and $\cH_{a,1}$ is isomorphic to $\alpha_{p,k}$.

\item If $a\neq 0$, $\cH_{a,2}$ is isomorphic to $\alpha_{p^2,k}$ if $p>2$, and to a twisted form of the $2$-torsion of a supersingular elliptic curve if $p=2$.

 \item
If $n\ge l\ge 0$, $\cD_l$  is a closed subgroup scheme of $\cD_n$, and if $l\ge 1$,  $\cD_{l}$ corresponds to  $\ker \fr^l_{\cD_n}$. 
    \end{enumerate1}

\end{lemm}






\begin{proof}
All the proofs, except $(5)$, are straightforward. 
If $char(k)>2$, $\cH_{a,2}$ is isomorphic to $\alpha_{p^2,k}$ via the isomorphism $\overline{T}\mapsto \overline{T}-a\frac{\overline{T}^{2p}}{2}$.
If $char(k)=2$ and $k$ algebraically closed, $\cH_{a,2}$ is isomorphic to $\cH_{1,2}$, via $\overline{T}\mapsto c \overline{T}$ where $c$ is a cubic root of $a$ (see also Lemma \ref{lem:isomorphisms H_al} for more details). And $\cH_{1,2}$ is  isomorphic to $\ker(\fr+V: W_{2,k}\to W_{2,k})$, where $W_{2,k}$ is the group scheme of Witt vectors of length $2$ and $V$ is the Verschiebung. This group scheme is known to be isomorphic to the $2$-torsion of a supersingular elliptic curve. 
\end{proof}

\begin{rema}
    If $n\ge 3$ and $a\neq 0$, $\cH_{a,n}$ is a subgroup scheme of the nonabelian extension of $\GG_a$ by $\GG_a$ given by the cocycle $a T^{p} T'^{p^{n-1}}$ (see \cite[II, \S 3, 4.6]{DG}).  
\end{rema}




\begin{lemm}\label{lem:subgroups Gn}
Let $n\ge 0$.  Any closed subgroup scheme of $\cD_n$ is equal to $\cD_l$, with $0\le l\le n$, 
or to $\cH_{a,m}$, for some $a\in k$ and $1\le m\le n$.  
\end{lemm}

\begin{proof}
The result is clear for $n \le 1$. So we suppose $n\ge 2$. In particular, $\cD_{n}$ is not commutative.
Let $H$ be a closed subgroup scheme of $\cD_n$. If $H\In \ker \fr^l_{\cD_n}$, for some $0\le l< n$, then $H$ is a closed subgroup scheme of $\cD_{l}$. Up to take minimal $0\le l \le n$ such that $\fr^{l}_H=0$, we can suppose that $\fr^{n-1}_H \neq 0$. In particular, if $H$ is a proper subgroup of $\cD_n$, $H$ has order $p^{n}$, otherwise iterating $n-1$ times the Frobenius, which has a kernel of order at least $p$, we will get the trivial morphism. 
Suppose that $H$ does not contain the center of $\cD_{n}$. Then the natural map $H\times \cH_{0,1}\to \cD_n$ is an isomorphism. Henceforth $H\simeq \cD_{n}/\cH_{0,1}\simeq \alpha_{p^{n-1},k}\times \alpha_{p,k}$, which would imply that $\cD_n$ is commutative. Therefore $H$ contains the center of $\cD_n$. As a consequence, $H/\cH_{0,1}$ is a closed subgroup scheme of $\cD_n/\cH_{0,1}\simeq \alpha_{p^{n-1},k}\times \alpha_{p,k}$. In particular, $H/\cH_{0,1}$ is a normal subgroup scheme of $\cD_n/\cH_{0,1}$, which implies that $H$ is a normal subgroup scheme of $\cD_n$.
So $H$ is obtained as the kernel of a morphism from $\cD_n$ to $\alpha_p$. Any such a morphism is given by an element
$$
P(\overline{S},\overline{T})=
\sum_{0\le i<p, 0\le j<p^{n}}a_{ij}\overline{S}^i\overline{T}^j\in k[S,T]/(S^{p},T^{p^{n}})$$
such that $P(\overline{S},\overline{T})^p=0$ and
\begin{equation}\label{eq: morphism to alpha_p}
\sum_{\substack{0\le i<p\\ 0\le j<p^{n}}}a_{ij}(\overline{S}\otimes 1+1\otimes\overline{S})^i(\overline{T}\otimes 1+ 1\otimes \overline{T}+\overline{T}^p\otimes \overline{S})^j=\sum_{\substack{0\le i<p\\ 0\le j<p^{n}}}a_{ij}(\overline{S}^i\overline{T}^j\otimes 1+1\otimes \overline{S}^i\overline{T}^j).
\end{equation}
in $
k[S,T]/(S^{p},T^{p^{n}})\otimes k[S,T]/(S^{p},T^{p^{n}}).
$ We can suppose $a_{00}=0$.
If we reduce modulo $(\overline{T})\otimes (1)$ we get
$$
\sum_{\substack{0\le i<p\\ 0\le j<p^{n}}}a_{ij}(\overline{S}\otimes 1+1\otimes\overline{S})^i( 1\otimes \overline{T}^j)=\sum_{\substack{0\le i<p\\ 0< j<p^{n}}}a_{ij}(1\otimes \overline{S}^i\overline{T}^j)+\sum_{0<i<p}a_{i0}(1\otimes \overline{S}^i+\overline{S}^i\otimes 1).
$$
in $k[S,T]/(S^{p},T)\otimes k[S,T]/(S^{p},T^{p^{n}})$. Therefore $a_{ij}=0$ if $1<i<p$ or $i=1$ and $j> 0$. 
So
$$
P(\overline{S},\overline{T})=a_{1,0}\overline{S}+Q(\overline{T}).
$$
If we reduce \eqref{eq: morphism to alpha_p} modulo $(1)\otimes (\overline{S})$ we find that $Q(\overline{T})$ should be additive. Since $P(\overline{S},\overline{T})^p=0$ we get that $P(\overline{S},\overline{T})=a_{10}\overline{S}+a_{0p^{n-1}}\overline{T}^{p^{n-1}}$. Moreover,  $a_{10}\neq 0$ since we supposed that $\fr^{n-1}_H\neq 0$.
So we have that $H$ is isomorphic to $\cH_{a,n}$, with $a=\frac{a_{0p^{n-1}}}{a_{10}}$.
\end{proof}


\begin{lemm}
Let $G$ be an infinitesimal unipotent group scheme with unidimensional Lie algebra over a field of characteristic $p$. An action, as an automorphism of groups, of an infinitesimal group scheme $H$ of multiplicative type over $G$ is faithful if and only if the induced action on $\ker \fr_G$ is faithful. And, if it happens, $H$ has unidimensional Lie algebra.
\end{lemm}
\begin{proof}
The 'if' part is obvious. We prove the 'only if' part.
We can suppose that $H$ is of height $1$ since the kernel of any action of an infinitesimal group scheme has a nontrivial intersection with the kernel of the Frobenius (see also \cite[Proposition 6.1]{gouthier2023infinitesimal}). We also remark that $\dim_k \Lie (H)=\dim_k \Lie (\ker{\fr_H})$.
Therefore, by \cite[III,\S 6, Propoition 7.1]{DG}, $H$ is a subgroup scheme of $\cA ut_1(G)=\ker (\cA ut(G)\to \cA ut(G/\ker \fr_G))$.
Moreover, $\ker \fr_G\simeq \alpha_p$ is contained in the center since a unipotent group scheme has a nontrivial center, and, since $G$ is infinitesimal unipotent, the intersection with the kernel of the Frobenius is nonempty.  
Therefore the induced action of $G/\ker F_G$ over $\ker F_G$ is trivial, then, by \cite[III,\S 6, Proposition 7.4]{DG}, we have an exact sequence
$$
0\to \cH om_{gr}(G/\ker \fr_G,\alpha_p)\too \cA ut_1(G)\too \cA ut_{gr}(\ker \fr_G)\simeq \GG_{m,k}.
$$
But we remark that $\cH om_{gr}(\alpha_p,\alpha_p)\simeq \GG_{a,k}$, so, by dévissage, we get that $\cH om_{gr}(G/\ker \fr_G,\alpha_p)$ is an unipotent group scheme.
Since $H$ is of multiplicative type, if $H$ acts faithfully on $G$ then it acts faithfully on $\ker \fr_G$. Therefore $H$ is isomorphic to $\mu_{p,k}$ and so its Lie algebra is unidimensional.


\end{proof}
 \begin{lemm} \label{lem:isomorphisms H_al}
 Let $k$ be a field of characteristic $p$ and 
      let $n> 2$ or $n=2$ and $p=2$.  
     \begin{enumerate1}

\item Any action of $\mu_{p,k}$ over $\cH_{0,n}\simeq \alpha_{p^n,k}$, as automorphism of groups, is conjugate to $v \cdot t= v^{i} t$, for some $0\le i\le p-1$. 
Therefore, for any nontrivial action of $\mu_{p,k}$ over $\alpha_{p^n,k}$, all semi-direct products  $\alpha_{p^n,k}\rtimes \mu_p$ are isomorphic.

    \item    If $a,b \in k\setminus \{0\}$, then $\cH_{a,n}$ is isomorphic to $\cH_{b,n}$ if and only if $b/a$ is a $(p^{n}+p-1)$-th power. 
\item There are no nontrivial actions, as automorphism of groups, of infinitesimal group schemes of multiplicative type over $\cH_{a,n}$, for any $a\in k\setminus \{0\}$. 

\end{enumerate1}
 \end{lemm}
\begin{proof}
$(1)$. It is easy to see that the map $\cA ut_{gr}(\GG_a) \to \cA ut_{gr}(\alpha_{p^n})$ admits a section. Therefore any action of $\mu_{p,k}$ over  $\alpha_{p^n,k}$ extends to an action over $\GG_{a,k}$. Now, by \cite[III, \S6, Corollaire 7.9]{DG}, we have that any action of $\mu_{p,k}$ on $\GG_{a,k}$ is given by
 $$
v \cdot x= v^ix+(v^i-1)\sum_{l=1}^sa_l x^{p^l}
 $$
 for some $s\ge 1$, $0\le i \le p-1 $ and $a_i\in k$ for any $1\le l\le s$. Therefore any action of $\mu_{p,k}$ on $\alpha_{p^n,k}$ is given by
 $$
v \cdot x= v^ix+(v^i-1)\sum_{l=1}^{p^n-1}a_l x^{p^l}
  $$
  for some $a_i\in k$ for any $1\le l\le p^{n-1}-1$ and $0\le i \le p-1$. But this action is conjugated to the action $v\cdot x=v^i x$ via the automorphism $x\mapsto x+\sum_{l=1}^{p^n-1}a_l x^{p^l}$. If the action is nontrivial, then $i>0$. 
 Moreover, $v\mapsto v^i$ is an automorphism of $\mu_{p,k}$, therefore all the associated semi-direct products are isomorphic.
 $(2),(3)$.  We recall that the Hopf algebra of $\cH_{a,n}$ is isomorphic to $k[T]/(T^{p^{n}})$,  where the comultiplication is given by $\overline{T}\otimes 1 + 1\otimes \overline{T}+a\overline{T}^{p^{n-1}}\otimes \overline{T}^{p}.$ We now consider an element of $\cI som(\cH_{a,n},\cH_{b,n})(R)$, with $R$ a $k$-algebra.
An isomorphism from $\cH_{a,n,R}$ to $\cH_{b,n,R}$ is given by an element $P(\overline{T})=\sum_{i=1}^{p^{n}-1}a_i\overline{T}^i\in R[T]/(T^{p^{n}})$ such that $a_1\in R$ is invertible and
$$
\sum_{i=1}^{p^{n}-1}a_i(\overline{T}^i\otimes 1+1\otimes \overline{T}^i) +a(\sum_{i=1}^{p-1}a_{i}^{p^{n}}\overline{T}^{ip^{n-1}})\otimes (\sum_{i=1}^{p^{n-2}}a_i^p\overline{T}^{ip})=\sum_{i=1}^{p^{n}-1}a_i(\overline{T}\otimes 1+1\otimes \overline{T}+b\overline{T}^{p^{n-1}}\otimes \overline{T}^p)^i.
$$
Since it induces an isomorphism on the kernels of the $p^{n-1}$-th power of the Frobenius, which are isomorphic to $\alpha_{p^n,R}$, we get that $a_{i}=0$
 if $1<i<p^{n-1}$ and 
$i$ is not a power of $p$. 
Moreover, $a_i=0$ if $i>p^{n-1}$ and $i$ is not divisible by $p$ (you can see it deriving both sides).
So we get
$$\sum_{r= p^{n-2}+1}^{p^{n-1}-1}a_{pr}(\overline{T}^{pr}\otimes 1+1\otimes \overline{T}^{{pr}})+a
 a_{1}^{p^{n-1}} \overline{T}^{p^{n-1}} \otimes (\sum_{r=0}^{n-2}a_{p^r}^p\overline{T}^{p^{r+1}})=$$$$\sum_{r=p^{n-2}+1}^{p^{n-1}-1}a_{pr}{(\overline{T}^p\otimes 1+1\otimes \overline{T}^p)}^r+a_{1}b \overline{T}^{p^{n-1}}
\otimes\overline{T}^p$$
If $n>2$ and $a\neq 0$, comparing the coefficients of $\overline{T}^p\otimes \overline{T}^{p^{n-1}}$, we get $a_{p^{n-1}+p}=0$ and, comparing the coefficients of $\overline{T}^{p^{n-1}}\otimes \overline{T}^{p}$, 
we get $aa_1^{p^{n-1}+p}=a_1 b$, which means  that $a_1$ is a $(p^{n-1}+p-1)$-th power of $b/a$. 
If $n=2$ and $p=2$ the above equality reduces to
$$
a a_1^{2+2}\overline{T}^p\otimes \overline{T}^p=a_1 b \overline{T}^p\otimes \overline{T}^p,
$$
so $a_1$ is a  cubic root of $b/a$.  If $R=k$ this proves that $\cH_{a,n}$ is isomorphic to $\cH_{b,n}$ if and only if $b/a$ is a $(p^{n}+p-1)$-th power.
But this also proves that, if $a\neq 0$, the image of the map $\cA ut(\cH_{a,n})\to \cA ut(\ker \fr_{\cH_{a,n}})$ is contained in $\mu_{p^{n-1}+p-1,k}$.
This implies, by the previous Lemma, that any infinitesimal group scheme of multiplicative type acts trivially on $\cH_{a,n}$ if $a\neq 0$.
 On the other hand, if $a_1\in k$ is such that $a_1^{p^n+p-1}= b/a$ then the polynomial $P(\overline{T})=a_1 \overline{T}$ gives an isomorphism for any $n\ge 1$ and any prime $p$.

\end{proof}
\section{Proof of the Theorem}

We start with a Lemma.
\begin{lemm} \label{lem: necessary conditions}
Let $k$ be a field of characteristic $2$ and $G$ be a unipotent subgroup scheme of $\PGL_{2,k}$. Then 
\begin{enumerate1}
\item
$\ker \fr_G$ is isomorphic to $\alpha_{2,k}$ or to $\alpha_{2,k}\times \alpha_{2,k}$; 
\item
$\IM \fr_G$ is isomorphic to $\alpha_{2^n,k}$, for some $n$.


\end{enumerate1}
\end{lemm}
\begin{proof}\leavevmode

\begin{enumerate}

\item 
We observe that in characteristic $2$ we have the following exact sequence of restricted $p$-algebras
$$
0\to \Lie(\GG_a^2) \to \mathfrak{pgl}_2 \to \Lie(\GG_m)\to 0.
$$
Therefore $\Lie(\ker \fr_G)$ is necessarily contained in $\Lie(\GG_a^2)$, which implies that
$\ker F_G$ is a subgroup scheme of $\GG_a^2$. So we are done.


\item Let $G$ be a unipotent subgroup scheme of $\PGL_{2,k}$ and let $\tilde{G}:=\pi^{-1}G$, where $\pi:\SL_{2,k}\to \PGL_{2,k}$ is the projection. Then 
we have an exact sequence
$$
1\to \mu_{2,k} \to \tilde{G}\to G \to 0
$$
which yields the following commutative diagram
$$
\xymatrix{
1\ar[r]& \mu_{2,k} \ar[r]\ar^{\fr_{\mu_{2,k}}}[d]& \tilde{G}\ar[r]\ar^{\fr_{\tilde{G}}}[d] &G \ar[r]\ar^{\fr_G}[d]& 0\\
1\ar[r]& \mu_{2,k} \ar[r]&\tilde{G}^{(p)}\ar[r] &G^{(p)} \ar[r]& 0
}
$$
where the vertical maps are the relative Frobenius.
Since the Frobenius is trivial on $\mu_{2,k}$, we get that $\fr_{G}$ factorizes as
$$
G\on{\alpha}\to \tilde{G}^{(p)} \to G^{(p)}.
$$
 Since $G$ is unipotent then $\alpha(G)$ is a unipotent subgroup of $\tilde{G}^{(p)}\In \SL_{2,k}$. Therefore  $\alpha(G)$ is isomorphic to $\alpha_{2^n}$, for some $n$. So the statement follows.

\end{enumerate}

\end{proof}

We now continue with the proof of the Theorem.
Let $G$ be a unipotent subgroup scheme of $\PGL_{2,k}$. 
Let us consider $\tilde{G}:=\pi^{-1}{G}\In \SL_{2,k}$. 
For any $n\ge 1$, we apply the Snake Lemma to this commutative diagram with exact rows
$$
\xymatrix{
1\ar[r]& \mu_{2,k} \ar[r]\ar^{\fr^n_{\mu_{2,k}}}[d]& \tilde{G}\ar[r]\ar^{\fr^n_{\tilde{G}}}[d] &G \ar[r]\ar^{\fr^n_G}[d]& 0\\
1\ar[r]& \mu_{2,k} \ar[r]&\tilde{G}^{(p^n)}\ar[r] &G^{(p^n)} \ar[r]& 0
}
$$
and we get an exact sequence
$$
0\to \mu_{2,k} \to \ker \fr^n_{\tilde{G}}\to \ker \fr^n_G \to 0,
$$
since the are no nontrivial morphisms from a unipotent group scheme to a diagonalizable group scheme. We also observe that $\ker \fr^n_{\tilde{G}}=\pi^{-1}(\ker {\fr^n_G})$.

Let us firstly suppose that $\dim_k \Lie G=2$.  Then  $\ker \fr_G$ has order $2^2$, 
so $\ker \fr_{\tilde{G}}$ has order $2^3$ and therefore it coincides with $\ker \fr_{\SL_{2,k}}$.
Now,  $\IM{\fr_G}=\IM{\fr_{\tilde{G}}}$, and, by Lemma  \ref{lem: necessary conditions}, they are both isomorphic to $\alpha_{2^n}$ for some $n\ge 0$. 
So we have the following commutative diagram with exact rows
$$
\xymatrix{
1\ar[r]& \ker {\fr_{\SL_{2,k}}} \ar[r]\ar^{\id}[d]& \tilde{G}\ar[r]\ar[d] &\alpha_{p^n} \ar[r]\ar[d]& 0\\
1\ar[r]& \ker \fr_{\SL_{2,k}}  \ar[r]&\ker \fr^{n+1}_{\SL_{2,k}} \ar^{\fr_{\SL_{2,k}}}[r] &\ker \fr^n_{\SL_{2,k}} \ar[r]& 0.
}
$$
The diagram is a pull-back since it is a morphism of extensions.
By Lemma \ref{prop: subgroups SL2}, any  unipotent subgroup  of $\ker \fr^n_{\SL_{2,k}}$ is
$$
H_{s_1,s_2,n}=\{\left( \begin{smallmatrix}  x_{11} &x_{12} \\ x_{21} & x_{11}  \end{smallmatrix} \right)\in \ker \fr^n_{\SL_{2,k}}| s_i (x_{11}-1)+s_j(x_{ij})=0, \text{ for } (i,j)=(1,2),(2,1) \}
$$
for some $[s_1:s_2]\in \PP^1(k)$.
We remark that the inverse is given by the identity.
We are going to prove that, for any $[s_1:s_2]\in \PP^1(k)$, 
$G=\fr_{\SL_{2,k}}^{-1}(H_{s_1,s_2,n})/\mu_{2,k}$
is isomorphic to $\cD_n$.  This would also prove the Theorem in the case  $\dim_k \Lie G=1$, since $\tilde{G}$ is contained in the pull-back $\fr_{\SL_{2,k}}^{-1}(H_{s_1,s_2,n})$, therefore $G$ would be a subgroup scheme of  $\fr_{\SL_{2,k}}^{-1}(H_{s_1,s_2,n})/\mu_{2,k}\simeq \cD_n$.
We firstly remark that $\fr_{\SL_{2,k}}^{-1}(H_{s_1,s_1,n})/\mu_{2,k}$ has order $p^{n+2}$.
The Hopf algebra of $\fr_{\SL_{2,k}}^{-1}(H_{s_1,s_2,n})$ is 
%
%
$k [X_{ij}]_{1\le i,j\le 2}$ quotiented by the ideal
$$
\bigg((X_{ii}-1)^{2^{n+1}},X_{ij}^{2^{n+1}},(X_{ii}-X_{jj})^{2},X_{{ii}}X_{jj}-X_{ij}X_{ji}-1, s_i (X_{ii}^2-1)+s_j(X_{ij}^2))\bigg)_{i\neq j}
$$
and the comultiplication is induced by
$$
X_{ij}\mapsto X_{i1}X_{1j}+X_{i2}X_{2j}
$$
for any $1\le i,j \le 2$.
Now we compute the invariant ring by the natural action of $\mu_{2,k}$.
 We let $Y_{ijkl}=\overline{X_{ij}X_{kl}}$. 
Consider the subalgebra $A$ of  $k[\fr_{\SL_{2,k}}^{-1}(H_{s_1,s_2})]$ generated by $Y_{1112}$ and $Y_{2122}$.  
%
%
We suppose $s_1\neq 0$. If not, $s_2\neq 0$ and a similar argument works.
Looking at the comultiplication we have that
\begin{align*}
Y_{1112}\mapsto & (\overline{X_{11}}\otimes \overline{X_{11}}+\overline{X_{12}}\otimes \overline{X_{21}})(\overline{X_{11}}\otimes \overline{X_{12}}+\overline{X_{12}}\otimes \overline{X_{22}})=\\ &Y_{1111}\otimes Y_{1112}+ Y_{1112}\otimes Y_{1122}+ Y_{1112}\otimes Y_{1221}+ Y_{1212}\otimes Y_{2122}=\\
&Y_{1111}\otimes Y_{1112}+ Y_{1112}\otimes( Y_{1122}+ Y_{1221})+ Y_{1112}^2\otimes s_1Y_{2122}=\\
&Y_{1112}\otimes 1+1\otimes Y_{1112}+    Y_{1112}^2\otimes(   s_1Y_{2122}+s_2Y_{1112})
\end{align*}

and

\begin{align*}
Y_{2122}\mapsto &(\overline{X_{21}}\otimes \overline{X_{11}}+\overline{X_{22}}\otimes \overline{X_{21}})(\overline{X_{21}}\otimes \overline{X_{12}}+\overline{X_{22}}\otimes \overline{X_{22}})=\\&Y_{2121}\otimes Y_{1112}+ Y_{2122}\otimes Y_{1122}+Y_{2122}\otimes Y_{1221}+ Y_{2222}\otimes Y_{2122}= \\
&Y_{2121}\otimes Y_{1112}+ Y_{2122}\otimes (Y_{1122}+ Y_{1221})+ Y_{2222}\otimes Y_{2122}=\\
&Y_{2212}\otimes 1+ 1\otimes Y_{2212}+ Y_{1112}^2\otimes (s_2 Y_{2122}+\frac{s_2^2}{s_1}Y_{1112})
\end{align*}
where we used that $Y_{1122}+Y_{1221}=1, s_1Y_{1112}^2=Y_{1212}$, $s_1Y_{2121}=s_2^2Y_{1112}^2$ and $Y_{2222}=Y_{1111}=1+s_2Y_{1112}^2$, which are easy to verify. The inverse map is the identity, so $A$ is an Hopf algebra contained in $k[(\fr_{SL_2}^{-1}(H_{s_1,s_2})]^{\mu_{2,k}}$.
We also stress that
$$
(s_1 Y_{2122}+s_2 Y_{1112}) \mapsto (s_1 Y_{2212}+s_2 Y_{1112}) \otimes 1 +1\otimes (s_1 Y_{2212 }+s_2 Y_{1112})
$$
and $s_1 Y_{2122}+s_2 Y_{1112}$ is not zero. Indeed for example the element $\left(\begin{smallmatrix}  1 &0  \\ a & 1  \end{smallmatrix} \right)$, with $a^2=0$ and $a\neq 0$, belongs to $\fr_{SL_2}^{-1}(H_{s_1,s_2})$.
This means that $s_1 Y_{2122}+s_2 Y_{1112}$ generates an Hopf algebra $B$ isomorphic to $\alpha_{2,k}$.
Therefore there is an epimorphism $H=\spec A \to \spec B$. The kernel of this map is the spectrum of the algebra generated by the class of $Y_{1112}$ modulo  $s_1 Y_{2122}+s_2 Y_{1112}$. This group scheme is isomorphic to $\alpha_{2^{n+1}}$, since $Y_{1112}^{2^{n}}=X_{11}^{2^n}Y_{12}^{2^n}\neq 0$. Here we use that $s_1 \neq 0$. If $s_1=0$, we would have that $Y_{2212}^{2^{n}} \neq 0$.
So $H$ has order $2^{n+2}$. Since it has the same order of $\fr_{SL_2}^{-1}(H_{s_1,s_2,n})$, the natural map $\fr_{SL_2}^{-1}(H_{s_1,s_2},n) \to H$ is an isomorphism.
Moreover $H$ is isomorphic to $\cD_n$ via the isomorphism
$$
\overline{S}\mapsto s_1 Y_{2122}+s_2 Y_{1112}
$$
and
$$
\overline{T}\mapsto Y_{2122}.
$$
Finally let us suppose that  $k$ is a perfect field. Then any trigonalizable group scheme $T$ is a semidirect product $U\rtimes D$, where $U$ is unipotent and $D$ is diagonalizable. If $T$ is infinitesimal and a closed subgroup scheme of $\PGL_{2,k}$ then, by \cite[Corollary 6.4]{gouthier2023infinitesimal}, $D\simeq \mu_{2^l,k}$, for some $l\ge 1$, $U$ is trivial or  the action of $\mu_{2^l,k}$ on $U$ is nontrivial. Therefore, by the first part of the Theorem and by the Lemma \ref{lem:isomorphisms H_al}$(3)$, $U$ is isomorphic to $\cD_n$ or to $\alpha_{2^n,k}$, for some $n\ge 1$.




\section{Actions of trigonalizable group schemes over $\PP^1_k$} \label{sec: trigonalizable}
In the Theorem we prove that if $k$ is a perfect field of characteristic $2$, any trigonalizable, not unipotent nor diagonalizable, infinitesimal subgroup scheme of $\PGL_{2,k}$ is a semi-direct product of $\mu_{2^l,k}$ by $\cD_n$ or $\alpha_{2^n}$, for some $n,l\ge 0$. 
In this section, we say something more about which semi-direct products can appear.

\begin{prop}
Let $k$ be a field of characteristic $2$.
There exists a nontrivial action, as an automorphism of groups, of $\GG_{m,k}$ on 
$\cE$ such that 
$\cE\rtimes \GG_{m,k}$ acts faithfully on $\PP^1_k$. Moreover, we can choose the action of $\GG_{m,k}$ such that, for any $n\ge 0$, it preserves $\cD_{n}$ and $\cH_{0,n}$. We then obtain, by restriction, faithful actions of the induced semi-direct products $\cD_n\rtimes \mu_{2^l,k}$ and $\cH_{0,n}\rtimes \mu_{2^l,k}$ over $\PP^{1}_k$, for any $l\ge 1$.
\end{prop}
\begin{proof}
    
Let us consider the action of $\GG_{m,k}=\spec k[V,V^{-1}]$ over $\cE$  given by $v\cdot (s,t)=(v^{-1}s,vt)$.
Then we define the action of $\cE\rtimes \GG_{m,k}$ over $\AA^1_k=\spec k[X]$ given by the coaction
$$
X\longmapsto V X +V^2 \overline{S}X^2+T.
$$
And we get
 $$
 \frac{1}{X}\longmapsto \frac{1}{T+VX}\Big(1+\frac{V^2  \overline{S}X^2}{T+ {V}X}\Big)= \overline{S}+\frac{1}{T+VX}+ \frac{T^2\overline{S}}{(T+VX)^2}\in k[X,T,V^{\pm 1}, (T+VX)^{-1}] [S]/(S^2),
 $$
 Therefore, by gluing,  we get a faithful action of  $\cE\rtimes \GG_{m,k}$ over $\PP^1_k$.

\end{proof}

    
\begin{rema}
\begin{enumerate1} \leavevmode
 \item The group scheme $\ker{\fr_{\PGL_{2,k}}}$ is isomorphic to the semidirect product $\alpha_{2,k}^2\rtimes \mu_{2,k}$, where the action is given by $v\cdot (s,t)=(vs,vt)$. Indeed it coincides with the kernel of the Frobenius of the group scheme constructed in the Proposition.
\item We remark that taking $V=1$ in the proof of the Proposition, we get an explicit action of any unipotent infinitesimal group scheme over $\PP^1_k$, since, as proved in the Theorem, they are all subgroup schemes of $\cE$. 

\end{enumerate1}
\end{rema}

We finally mention the following Lemma.
\begin{lemm}
Let $k$ be a perfect field of characteristic $2$.
Up to isomorphism, there is a unique infinitesimal trigonalizable subgroup scheme of $\PGL_{2,k}$ extension of $\mu_2$ by $\alpha_{2^n}$, for any $n>0$. 
\end{lemm}
\begin{proof}
    This follows by the Theorem and by Lemma \ref{lem:isomorphisms H_al}$(1)$.
\end{proof}

\bibliographystyle{amsalpha}
\bibliography{bib}

\vspace{2cm}

\noindent Institut de Math\'ematiques de Bordeaux, 351 Cours de la Lib\'eration, 33405 Talence, France \\ {Email Addresses:} \texttt{bianca.gouthier@math.u-bordeaux.fr},
 \texttt{dajano.tossici@math.u-bordeaux.fr}
\\
\end{document}